\documentclass[12pt,reqno]{amsart}
\usepackage{amsmath,amssymb,amsfonts,amscd,latexsym,amsthm,mathrsfs,verbatim,comment}
\usepackage{enumerate}
\usepackage{tikz}
\usepackage[unicode]{hyperref}
\newtheorem{theorem}{Theorem}

\theoremstyle{remark}
\newtheorem{remark}{\bf Remark}
\newtheorem*{acknowledgements}{\bf Acknowledgements}
\renewcommand{\d}{{\mathrm d}}

\newcommand{\Z}{\mathbb{Z}}

\begin{document}

\title{An integrality phenomenon}

\author{Florian F\"urnsinn}
\address{University of Vienna, Faculty of Mathematics, Vienna, Austria}
\email{florian.fuernsinn@univie.ac.at}

\author{Danylo Radchenko}
\address{Institut des Hautes \'Etudes Scientifiques, CNRS, Laboratoire Alexandre Grothendieck, 35 route de Chartres, Bures-sur-Yvette 91440, France}
\email{danradchenko@gmail.com}

\author{Wadim Zudilin}
\address{Department of Mathematics, IMAPP, Radboud University, PO Box 9010, 6500 GL Nij\-me\-gen, Netherlands}
\email{wzudilin@gmail.com}

\date{19 April 2026}


\begin{abstract}
We prove a general statement about the integrality of the sequences generated by a recursion of the following form: $n u_n$ equals a linear combination of $u_{n-1},u_{n-2},\dots,u_0$ with polynomial coefficients in $n$ of special form.
This includes a conjectural integrality of the sequence related to the H\"ormander--Bernhardsson extremal function, for which we further give a direct proof as well.
\end{abstract}

\maketitle

The Ap\'ery numbers $a_0,a_1,a_2,\dots$ appear as the denominator sequence of Ap\'ery's approximations to $\zeta(3)$ \cite{Ap79,Po79}; they can be generated via the recursion
\[
n^3a_n=(2n-1)(17n^2-17n+5)a_{n-1}-(n-1)^3a_{n-2} \quad\text{for}\; n>0,
\quad a_0=1, \; a_{-1}=0.
\]
The fact that all elements in the sequence are integers is quite remarkable, since one needs to divide at each step by $n^3$ to generate the consecutive term $a_n$, hence $a_n$ is likely to be rational with denominator `close' to~$n!^3$.
One way to resolve the integrality mystery is to show that $a_n=\sum_{k=0}^n\binom{n+k}{k}^2\binom{n}{k}^2$; for a discussion of other strategies to do this see \cite{ASZ11,Za09} and references therein.
The integrality phenomenon is not common for recursions like Ap\'ery's (see~\cite{Za09}); the literature mainly discusses the cases when the resulting generating function $\sum_{n=0}^\infty a_nx^n$ satisfies a linear differential equation of Picard--Fuchs type, that is, a differential equation describing the variation of periods of a family of algebraic varieties defined over a number field. 
In this note we give a family of recursions of a very different type that exhibit the same integrality phenomenon.

\section{An irregular example of integrality}
\label{s1}
Investigation of the H\"ormander--Bernhardsson extremal function in \cite{BORS25} brought to light a different type of example that originates from a differential equation with irregular singularities,
\begin{equation}
x^4\,\frac{\d^2y}{\d x^2}+2x^3\,\frac{\d y}{\d x}
-\big(\alpha^2x^4+(\beta^2-\tfrac14)x^2+\gamma^2\big)y=0,
\label{D2a}
\end{equation}
where $\alpha,\beta,\gamma$ are parameters.
Equation \eqref{D2a} possesses analytic solutions at neither $x=0$ nor $x=\infty$.
The substitution $x=2\gamma t$ brings it to the form
\begin{equation}
t^4\,\frac{\d^2y}{\d t^2}+2t^3\,\frac{\d y}{\d t}
-\big(4(\alpha\gamma)^2t^4-(\tfrac12+\beta)(\tfrac12-\beta)t^2+\tfrac14\big)y=0;
\label{D2b}
\end{equation}
the latter has formal solutions $e^{-1/(2t)}g(t)$ and (because of the symmetry $t\mapsto-t$ of the equation) $e^{1/(2t)}g(-t)$, with $g(t)=g(t;\alpha,\beta,\gamma)$ a power series at $t=0$ satisfying
\begin{equation}
t^2\,\frac{\d^2g}{\d t^2}+(1+2t)\frac{\d g}{\d t}
-\big(4(\alpha\gamma)^2t^2-(\tfrac12+\beta)(\tfrac12-\beta)\big)g=0.
\label{D2c}
\end{equation}
If we write $g(t)=\sum_{n=0}^\infty w_nt^n$, then the differential equation \eqref{D2c} translates into the recursion
\begin{equation}
\begin{gathered}
(n+1)w_{n+1}+(n+\tfrac12+\beta)(n+\tfrac12-\beta)w_n-4(\alpha\gamma)^2w_{n-2}=0
\quad\text{for}\; n=0,1,2,\dots,
\\ w_0=1, \;\; w_{-1}=w_{-2}=0,
\end{gathered}
\label{w-rec}
\end{equation}
for the coefficients; $g(0)=w_0=1$ fixes the normalisation of $g(t)$.

The symmetric square of the differential equation \eqref{D2b} reads
\begin{equation}
t^4\,\frac{\d^3Y}{\d t^3}+6t^3\,\frac{\d^2Y}{\d t^2}
-\big(16(\alpha\gamma)^2t^4-(7-4\beta^2)t^2+1\big)\frac{\d Y}{\d t}
-4t\big(8(\alpha\gamma)^2t^2-(\tfrac12+\beta)(\tfrac12-\beta)\big)Y=0.
\label{D3}
\end{equation}
Its solution space is spanned by $\big(e^{-1/(2t)}g(t)\big)^2$, $\big(e^{1/(2t)}g(-t)\big)^2$ and $e^{-1/(2t)}g(t)\cdot e^{1/(2t)}g(-t)$; while in general all three solutions are formal, the latter $G(t)=g(t)g(-t)$ is an even power series at the origin:
\[
G(t)=\sum_{n=0}^\infty u_nt^{2n},
\]
where $u_n=\sum_{k=0}^{2n}(-1)^kw_kw_{2n-k}$ for $n=0,1,2,\dots$\,.
Furthermore, the differential equation \eqref{D3} translates into the recursion
\begin{equation}
\begin{gathered}
(n+1)u_{n+1}-2(2n+1)(n+\tfrac12+\beta)(n+\tfrac12-\beta)u_n+16(\alpha\gamma)^2nu_{n-1}=0
\quad\text{for}\; n=0,1,2,\dots,
\\ u_0=1, \;\; u_{-1}=0,
\end{gathered}
\label{u-rec}
\end{equation}
for the coefficients of $G(t)$.

We can further restate the equations in \eqref{D2b}--\eqref{u-rec} in terms of just two parameters $c=4(\alpha\gamma)^2$ and $b=\beta^2-\frac14$;
for example, the recursions \eqref{w-rec} and \eqref{u-rec} assume the forms
\begin{equation}
nw_{n}+(n(n-1)-b)w_{n-1}-cw_{n-3}=0
\quad\text{for}\; n=0,1,2,\dots
\label{w-rec1}
\end{equation}
and
\begin{equation}
nu_{n}-2(2n-1)(n(n-1)-b)u_{n-1}+4c(n-1)u_{n-2}=0
\quad\text{for}\; n=0,1,2,\dots\,.
\label{u-rec1}
\end{equation}
It has been numerically observed in \cite[Conjecture 2]{BORS25} that the terms $u_n=u_n(b,c)$ generated by the latter recurrence equation starting from $u_0=1$ (and $u_{-1}=0$) are polynomials in $b,c$ with \emph{integer} coefficients; note that \eqref{u-rec1} as in Ap\'ery's case suggests only $n!u_n\in\mathbb Z[b,c]$.
In contrast, the coefficients of polynomials $w_n=w_n(b,c)$ generated by \eqref{w-rec1} are very far from being integral (or 2-integral): from reading off the coefficient of $b^n$ in $w_n$ one finds out immediately that the inclusion $n! w_n\in\mathbb Z[b,c]$ is the best possible.

\begin{theorem}
\label{ex-conj}
The sequence $u_n$ generated by the recursion \eqref{u-rec1} and initial data $u_0=1$, $u_{-1}=\nobreak0$, is integer-valued: $u_n\in\mathbb Z[b,c]$.
\end{theorem}

Before proceeding with the proof, consider the special choice $c=0$ in which the recurrence equations \eqref{w-rec1} and \eqref{u-rec1} reduce to two-term recursions; those can be then solved explicitly and we arrive at the equality
\begin{equation}
\sum_{n=0}^\infty(\tfrac12+\beta)_n(\tfrac12-\beta)_n\frac{t^n}{n!}
\cdot\sum_{n=0}^\infty(\tfrac12+\beta)_n(\tfrac12-\beta)_n\frac{(-1)^nt^n}{n!}
=\sum_{n=0}^\infty(\tfrac12+\beta)_n(\tfrac12-\beta)_n\binom{2n}{n}t^{2n},
\label{eq:HG}
\end{equation}
where the Pochhammer notation $(\alpha)_n=\Gamma(\alpha+n)/\Gamma(\alpha)=\prod_{i=0}^{n-1}(\alpha+i)$ is used.
The identity can be stated and proved hypergeometrically; it is related to (a special case of) the classical Clausen's identity
\[
\bigg(\sum_{n=0}^\infty\frac{(\tfrac12+\beta)_n(\tfrac12-\beta)_n}{n!^2}t^n\bigg)^2
=\sum_{n=0}^\infty\frac{(\tfrac12+\beta)_n(\tfrac12-\beta)_n}{n!^2}\binom{2n}{n}\big(t(1-t)\big)^n,
\]
and for the latter several non-hypergeometric deformations are known \cite{ASZ11,CTYZ11}.

\begin{proof}[Proof of \autoref{ex-conj}]
Our deformation of identity \eqref{eq:HG} reads
\begin{equation}
\sum_{n=0}^\infty u_nt^{2n}
=\sum_{k=0}^\infty(-1)^kc^kt^{4k}\sum_{n=0}^\infty n!w_n\binom{n+k}{k}\binom{2n+2k}{n+k}(-1)^nt^{2n}.
\label{id3}
\end{equation}
Once written, verification that both sides satisfy the same linear differential equation is straightforward.
Identity \eqref{id3} leads to
\begin{equation}
u_n
=(-1)^n\sum_{k=0}^{\lfloor n/2\rfloor}(-1)^kc^k(n-2k)!w_{n-2k}\binom{n-k}k\binom{2n-2k}{n-k}
\quad\text{for}\; n=0,1,2,\dots\,,
\label{bin}
\end{equation}
from which the integrality of $u_n$ is transparent in view of $n!w_n\in\mathbb Z[b,c]$.
\end{proof}

\begin{remark}
\label{R1}
One can invert formula \eqref{bin}:
\begin{equation}
n!\,w_n
=\sum_{m=0}^n\binom{n}{m}{\binom{2m}m}^{-1}
\sum_{k=0}^m(-1)^{n+m+k}\bigg(\binom{2m}{m-k}-\binom{2m}{m-k-1}\bigg)2^{m-k}c^{(n-k)/2}u_k
\label{inv}
\end{equation}
for $n=0,1,2,\dots$; \emph{a posteriori} only $k\equiv n\bmod2$ appear in the latter sum.
\end{remark}

\begin{remark}
\label{R2}
One can also cast the transformation in \eqref{id3} as
\begin{align*}
\sum_{n=0}^\infty u_nt^{2n}
=\frac1{\sqrt{1+4ct^4}}
\sum_{n=0}^\infty\binom{2n}{n}n!w_n\bigg(\frac{-t^2}{1+4ct^4}\bigg)^n.
\end{align*}
\end{remark}

\section{A general integrality statement}
\label{s2}
The following result shows that~\eqref{u-rec1} belongs to a broad family of recursions exhibiting an analogous integrality property.

\begin{theorem}
	\label{thmrec}
	Let $R$ be an integral domain of characteristic zero with fraction field $K$, and let $p_i\in tR[t^2]$, $i\ge1$, be a sequence of odd polynomials. Define $\{u_n\}_{n\ge0}\subseteq K$ by the recursion
	\begin{equation}
		nu_n = \sum_{i=1}^{n}p_{i}(n-i/2)\,u_{n-i},
		\quad\text{for}\;n=1,2,\dots,
		\label{u-new2}
	\end{equation}
	with $u_0=1$. Then $u_n\in R[\frac12]$ for all $n\ge0$.
\end{theorem}

For a $d$-tuple of integral polynomials $\mathcal{Q}=(Q_1,\dots,Q_d)\in \Z[t]^d$, we recursively define a collection of rational functions $\langle \mathcal{Q}\rangle_{m}\in \mathbb{Q}(x_1,\dots,x_d)$, where $m=(m_1,\dots,m_d)\in\Z_{\ge0}^d$, by
	\[\langle \mathcal{Q}\rangle_{0}=1 \quad\text{and}\quad
	    \langle \mathcal{Q}\rangle_{m} = \frac{\sum_{i=1}^{d}Q_i(\langle m-e_i/2,x\rangle) \langle \mathcal{Q}\rangle_{m-e_i}}{\langle m, x\rangle} \quad\text{for}\; m\in \Z_{\ge0}^d\smallsetminus \{0\},\]
extended by $\langle\mathcal{Q}\rangle_{m}=0$ if $m\notin\Z_{\ge0}^d$.
Here $\langle u,v\rangle=u_1v_1+\dots+u_dv_d$, $x=(x_1,\dots,x_d)$, and $e_1,\dots,e_d$ are the standard basis vectors in~$\Z^d$.

\autoref{thmrec} easily follows from the following.
\begin{theorem} \label{thmpoly}
	For any odd polynomials $Q_1,\dots,Q_d\in t\Z[t^2]$ the elements $\langle \mathcal{Q}\rangle_{m}$ lie in $\Z[\frac12][x]$.
\end{theorem} 

\begin{proof}[Proof of \autoref{thmrec}]
	Write $p_i=\sum_{j}a_{ij}Q_{ij}$, for some $a_{ij}\in R$ and $Q_{ij}\in\mathbb{Z}[t]$. Expanding the recursion \eqref{u-new2} we see that $u_n$ is a weighted homogeneous polynomial of degree $n$ in $a_{ij}$ (with weight $i$ assigned to~$a_{ij}$), while the monomial $a_{i_1,j_1}^{n_1}\cdots a_{i_k,j_k}^{n_k}$, for distinct $(i_1,j_1),\dots,(i_k,j_k)$ and $n_1i_1+\dots+n_ki_k=n$, appears with the coefficient
	\[\langle Q_{i_1,j_1},\dots,Q_{i_k,j_k}\rangle_{n_1,\dots,n_k}(i_1,\dots,i_k).\]
	By~\autoref{thmpoly}, this coefficient lies in $R[\frac{1}{2}]$, hence also $u_n\in R[\frac{1}{2}]$.
\end{proof}

\begin{proof}[Proof of \autoref{thmpoly}]
	Let 
	\[F=F(x;z)=\sum_{m_1,\dots,m_d\ge0}\langle\mathcal{Q}\rangle_mz^m \in \mathbb{Q}(x_1,\dots,x_d)[[z_1,\dots,z_d]]\]
	be the generating function of $\langle\mathcal{Q}\rangle_m$. Let $\delta$ be the derivation
$\sum_{i=1}^dx_iz_i\frac{\partial}{\partial z_i}$. The recursive definition of $\langle\mathcal{Q}\rangle_m$ directly implies that 
	\begin{align*}
\delta F = \sum_{m}\langle m,x\rangle  \langle \mathcal{Q}\rangle_{m} z^m 
&= \sum_{i=1}^{d}z_i\sum_{m} \langle\mathcal{Q}\rangle_{m-e_i}Q_i(\langle m-e_i/2,x\rangle)z^{m-e_i} \\
&= \sum_{i=1}^{d}z_iQ_i(\delta+x_i/2)\,F 
= \sum_{i=1}^{d}z_i^{1/2}Q_i(\delta)z_i^{1/2}F.
\end{align*}
	  Multiplying both sides by $F$ we see that
	  \[\frac12\delta(F^2) = \sum_{i=1}^{d}(z_i^{1/2}F)Q_i(\delta)(z_i^{1/2}F).\]
	  Since $Q_i$ are odd polynomials and in view of the identity 
	  \[f\cdot \partial^{2k+1}f = \frac12\,\partial\bigg(\sum_{j=0}^{2k}(-1)^j(\partial^jf)\,(\partial^{2k-j}f)\bigg)\]
	  for any derivation $\partial$, for some quadratic polynomials $P_i(f)\in \Z[f,\partial f,\partial^2f,\dots]$ the following equality holds:
	  \[\delta\bigg(F^2 - \sum_{i=1}^{d}P_i(z_i^{1/2}F)\bigg)=0.\]
	  Equivalently, for some other quadratic polynomials $\widetilde P_i=\widetilde P_i(F)$ in $F,\delta F,\delta^2F,\dots$ with $\Z[\frac12][x]$ coefficients we have
	  \[\delta\bigg(F^2 - \sum_{i=1}^{d}z_i\widetilde P_i(F)\bigg)=0.\]
Integrating the latter we get that $F$ satisfies
	  \[F^2 = 1 + \sum_{i=1}^{d}z_i\widetilde P_i(F).\]
	  The claim is now evident by induction on $m_1+\dots+m_d$: the coefficient of $z^m$ on the left-hand side is
\[
2\langle\mathcal{Q}\rangle_{m}+\sum_{\substack{m'+m''=m\\m',m''\in\Z_{\ge0}^d\smallsetminus \{0\}}}\langle\mathcal{Q}\rangle_{m'}\langle\mathcal{Q}\rangle_{m''},
\]
while the right hand side is a $\Z[\frac12][x]$-polynomial expression in $\langle\mathcal{Q}\rangle_{n}$ with $n_i\le m_i$ and $n_1+\dots+n_d<m_1+\dots+m_d$.
\end{proof}

\begin{remark}
\label{R3}
Observe that when all $Q_i=t$ for $i=1,\dots,d$, we obtain $F^2=1+\sum_{i=1}^{d}z_iF^2$ and the generating function becomes explicit: $F=(1-z_1-\dots-z_d)^{-1/2}$.
Furthermore, note that the power of $2$ in the denominator of $\langle \mathcal{Q}\rangle_m$ grows at most linearly in $m_1+\dots+m_d$.
\end{remark}

\begin{remark}
\label{R4}
In the case of the recursion
\[
n^ka_n=\sum_{i=1}^nq_i(n-i/2)\,a_{n-i} \quad\text{for}\; n=1,2,\dots, \quad a_0=1,
\]
with $k>0$ odd and $q_i(t)\in t\mathbb Z[t^2]$, the substitution $a_n=u_n/n!^{k-1}$ and application of \autoref{thmrec} imply the $\mathbb Z[\frac12]$-integrality of $n!^{k-1}a_n$.
In particular, a sequence $a_n$ generated by the recursion
\[
n^3a_n=(2n-1)(An^2-An+B)a_{n-1}-C(n-1)^3a_{n-2} \quad\text{for}\; n>0,
\quad a_0=1, \; a_{-1}=0
\]
(compare with our motivating example from the introduction), where $A,B,C\in\mathbb Z$, satisfies $2^nn!^2a_n\in\mathbb Z$ rather than $n!^3a_n\in\mathbb Z$.
\end{remark}

\begin{remark}
\label{R5}
Regarding the timeline of this project, it started at the MPIM Bonn with discussions about the (somewhat unusual) integrality observed in \cite[Conjecture 2]{BORS25} and the proof, given now in Section~\ref{s1}, followed shortly after\,---\,all in May 2025. Then the second author found a general integrality conjecture, proving which required a considerable effort and a completely different set of tools. While finalising our project we learned about the preprint \cite{ND26} of Nguyen-Dang who independently worked on \cite[Conjecture 2]{BORS25}.
\end{remark}

\begin{acknowledgements}
We are thankful to Don Zagier for suggesting to state \autoref{thmrec} in its current form and for numerous comments on this note.
We thank the Max Planck Institute for Mathematics (Bonn, Germany) for warm hospitality extended to the authors in May--June 2025.

FF was funded by a DOC Fellowship (27150) of the Austrian Academy of Sciences at the University of Vienna; he further acknowledges financial support from the French--Austrian project EAGLES (ANR-22-CE91-0007 \& FWF grant 10.55776/I6130).
DR acknowledges funding by the European Union (ERC, FourIntExP, 101078782). WZ acknowledges support in part from the NWO grant OCENW.M.24.112.
\end{acknowledgements}

\end{document}